\UseRawInputEncoding
\documentclass{article}

\usepackage{stmaryrd,graphicx}
\usepackage{amssymb,mathrsfs,amsmath,amscd,amsthm,color}
\usepackage{tikz-cd}
\usepackage{float}
\usepackage{caption}
\usepackage{mathabx}
\usepackage[all,cmtip]{xy}
\DeclareMathAlphabet{\mathpzc}{OT1}{pzc}{m}{it}
\usepackage{amsfonts,latexsym,wasysym}
\usepackage[shortlabels]{enumitem}
\setlist{nosep}
\setenumerate{label=(\arabic*)}
\renewcommand{\int}{\operatorname{int}}

\newcommand{\im}{\operatorname{Im}}

\renewcommand{\span}{\pi_{n}^{Sp}}

\newcommand{\ui}{I}

\newcommand{\bfx}{\mathbf{x}}

\newcommand{\scrw}{\mathscr{W}}

\newcommand{\scru}{\mathscr{U}}

\newcommand{\scrv}{\mathscr{V}}

\newcommand{\bbe}{\mathbb{E}}

\newcommand{\bbn}{\mathbb{N}}

\newcommand{\bbr}{\mathbb{R}}

\newcommand{\bbz}{\mathbb{Z}}

\newcommand{\scry}{\mathscr{Y}}

\newcommand{\bspaces}{\mathbf{Top_{\ast}}}
\newcommand{\grp}{\mathbf{Grp}}
\newcommand{\ab}{\mathbf{Ab}}

\newcommand{\St}{\text{St}}
\newcommand{\st}{\text{st}}
\newcommand{\sd}{\text{sd}}
\renewcommand{\int}{\text{int}}
\newtheorem{theorem}{Theorem}[section]
\newtheorem{lemma}[theorem]{Lemma}

\newtheorem{corollary}[theorem]{Corollary}
\theoremstyle{definition}\newtheorem{definition}[theorem]{Definition}
\newtheorem{example}[theorem]{Example}

\newtheorem{remark}[theorem]{Remark}

\begin{document}
\title{Elements of higher homotopy groups undetectable by polyhedral approximation}
\author{John K. Aceti, Jeremy Brazas}
\date{\today}

\maketitle

\begin{abstract}
When non-trivial local structures are present in a topological space $X$, a common approach to characterizing the isomorphism type of the $n$-th homotopy group $\pi_n(X,x_0)$ is to consider the image of $\pi_n(X,x_0)$ in the $n$-th \v{C}ech homotopy group $\check{\pi}_n(X,x_0)$ under the canonical homomorphism $\Psi_{n}:\pi_n(X,x_0)\to  \check{\pi}_n(X,x_0)$. The subgroup $\ker(\Psi_n)$ is the obstruction to this tactic as it consists of precisely those elements of $\pi_n(X,x_0)$, which cannot be detected by polyhedral approximations to $X$. In this paper, we use higher dimensional analogues of Spanier groups to characterize $\ker(\Psi_n)$. In particular, we prove that if $X$ is paracompact, Hausdorff, and $LC^{n-1}$, then $\ker(\Psi_n)$ is equal to the $n$-th Spanier group of $X$. We also use the perspective of higher Spanier groups to generalize a theorem of Kozlowski-Segal, which gives conditions ensuring that $\Psi_{n}$ is an isomorphism.
\end{abstract}

\section{Introduction}

When non-trivial local structures are present in a topological space $X$, a common approach to characterizing the isomorphism type of $\pi_n(X,x_0)$ is to consider the image of $\pi_n(X,x_0)$ in the $n$-th \v{C}ech (shape) homotopy group $\check{\pi}_n(X,x_0)$ under the canonical homomorphism $\Psi_{n}:\pi_n(X,x_0)\to  \check{\pi}_n(X,x_0)$. The \textit{$n$-th shape kernel} $\ker(\Psi_n)$ is the obstruction to this tactic as it consists of precisely those elements of $\pi_n(X,x_0)$, which cannot be detected by polyhedral approximations to $X$. This method has proved successful in many situations for both the fundamental group \cite{CConedim,EK98onedimfgs,FG05,FZ05} and higher homotopy groups \cite{BrazSequentialnconn,EK00higher,EKcones,EKRZSnake, Kawamurasuspensions}. In this paper, we study the map $\Psi_{n}$ and give a characterization the $n$-th shape kernel in terms of higher-dimensional analogues of Spanier groups.

The subgroups of fundamental groups, which are now commonly referred to as ``Spanier groups," first appeared in E.H. Spanier's unique approach to covering space theory \cite{Spanier66}. If $\scru$ is an open cover of a topological space $X$ and $x_0\in X$, then the \textit{Spanier group with respect to }$\scru$ is the subgroup $\pi^{Sp}_{1}(\scru,x_0)$ of $\pi_1(X,x_0)$ generated by path-conjugates $[\alpha][\gamma][\alpha]^{-1}$ where $\alpha$ is a path starting at $x_0$ and $\gamma$ is a loop based at $\alpha(1)$ with image in some element of $\scru$. These subgroups are particularly relevant to covering space theory since, when $X$ is locally path-connected, a subgroup $H\leq \pi_1(X,x_0)$ corresponds to a covering map $p:(Y,y_0)\to (X,x_0)$ if and only if $\pi^{Sp}_{1}(\scru,x_0)\leq H$ for some open cover $\scru$ \cite[2.5.12]{Spanier66}. The intersection $\pi^{Sp}_{1}(X,x_0)=\bigcap_{\scru}\pi^{Sp}_{1}(\scru,x_0)$ is called the \textit{Spanier group of }$(X,x_0)$ \cite{FRVZ11}. The inclusion $\pi^{Sp}_{1}(X,x_0)\subseteq \ker(\Psi_{1})$ always holds \cite[Prop. 4.8]{FZ07}. It is proved in \cite[Theorem 6.1]{BFThickSpan} that $\pi^{Sp}_{1}(X,x_0)=\ker(\Psi_{1})$ whenever $X$ is paracompact Hausdorff and locally path connected. The upshot of this equality is having a description of level-wise generators (for each open cover $\scru$) whereas there may be no readily available generating set for the kernel of a homomorphism induced by a canonical map from $X$ to the nerve $|N(\scru)|$. Indeed, $1$-dimensional Spanier groups have proved useful in persistence theory \cite{Virkpersistence}. Since much of applied topology is based on a geometric refinement of polyhedral approximation from shape theory, there seems potential for higher dimensional analogues to be useful as well.

Higher dimensional analogues of Spanier groups recently appeared in \cite{BKPnspanier} and are defined in a similar way: $\pi_{n}^{Sp}(\scru,x_0)$ is the subgroup of $\pi_n(X,x_0)$ consisting of homotopy classes of path-conjugates $\alpha\ast f$ where $\alpha$ is a path starting at $x_0$ and $f:S^n\to X$ is based at $\alpha(1)$ with image in some element of $\scru$. Then $\pi_{n}^{Sp}(X,x_0)$ is the intersection of these subgroups. In this paper, we prove a higher-dimensional analogue of the $1$-dimensional equality $\pi^{Sp}_{1}(X,x_0)=\ker(\Psi_{1})$ from \cite{BFThickSpan}.

A space $X$ is $LC^n$ if for every neighborhood $U$ of a point $x\in X$, there is a neighborhood $V$ of $x$ in $U$ such that every map $f:S^k\to V$, $0\leq k\leq n$ is null-homotopic in $U$. When a space is $LC^{n}$ ``small" maps on spheres of dimension $\leq n$ contract by null-homotopies of relatively the same size. Certainly, every locally $n$-connected space is $LC^n$. However, when $n\geq 1$, the converse is not true even for metrizable spaces. Our main result is the following.

\begin{theorem}\label{maintheorem}
Let $n\geq 1$ and $x_0\in X$. If $X$ is paracompact, Hausdorff, and $LC^{n-1}$, then $\pi_{n}^{Sp}(X,x_0)=\ker (\Psi_n)$.
\end{theorem}

This result confirms that higher Spanier groups, like their $1$-dimensional counterparts, often identify precisely those elements of $\pi_n(X,x_0)$ which can be detected by polyhedral approximations to $X$. More precisely, under the hypotheses of Theorem \ref{maintheorem}, $g\in \span(X,x_0)$ if and only if $f_{\#}(g)=0$ for every map $f:X\to K$ to a polyhedron $K$. A first countable path-connected space is $LC^0$ if and only if it is locally path connected. Hence, in dimension $n=1$, Theorem \ref{maintheorem} only expands \cite[Theorem 6.1]{BFThickSpan} to some non-first countable spaces.

Regarding the proof of Theorem \ref{maintheorem}, the inclusion $\pi_{n}^{Sp}(X,x_0)\subseteq \ker (\Psi_n)$ was first proved for $n=1$ in \cite[Prop. 4.8]{FZ07} and for $n\geq 2$ in \cite[Theorem 4.14]{BKPnspanier}. We include this proof for the sake of completion (Lemma \ref{firstinclusioncor}). The proof of the inclusion $\ker (\Psi_n)\subseteq \pi_{n}^{Sp}(X,x_0)$ appears in Section \ref{section5mainresult} and is more intricate, requiring a carefully chosen sequence of open cover refinements using the $LC^{n-1}$ property. These refinements allow one to recursively extend maps on simplicial complexes skeleton-wise. These extension methods, established in Section \ref{section4extensions}, are similar to methods found in \cite{KSextension,KSvietoris}. 

We also put these extension methods to work in Section \ref{section6isomorphism} where we identify conditions that imply $\Psi_{n}$ is an isomorphism. In \cite{KSvietoris}, Kozlowski-Segal prove that if $X$ is paracompact Hausdorff and $LC^n$, then $\Psi_{n}$ is an isomorphism. In \cite{FZ07}, Fischer and Zastrow generalize this result in dimension $n=1$ by replacing ``$LC^1$" with ``locally path connected and semilocally simply connected." Similar, to the approach of Fischer-Zastrow, our use of Spanier groups shows that the existence of \textit{small} null-homotopies of small maps $S^n\to X$ (specifically in dimension $n$) is not necessary to prove that $\Psi_n$ is injective. We say a space $X$ is \textit{semilocally $\pi_n$-trivial} if for every $x\in X$ there exists an open neighborhood $U$ of $x$ such that every map $S^n\to U$ is null-homotopic in $X$. This definition is independent of lower dimensions but certainly $LC^{n}$ $\Rightarrow$ ($LC^{n-1}$ and semilocally $\pi_n$-trivial). Our second result proves Kozlowski-Segal's theorem under a weaker hypothesis and is stated as follows.

\begin{theorem}\label{mainresult2}
Let $n\geq 1$ and $x_0\in X$. If $X$ is paracompact, Hausdorff, $LC^{n-1}$, and semilocally $\pi_n$-trivial, then $\Psi_{n}:\pi_n(X,x_0)\to \check{\pi}_n(X,x_0)$ is an isomorphism.
\end{theorem}

The hypotheses in Theorem \ref{mainresult2} are the homotopical versions of the hypotheses used in \cite{Mardesic} to ensure that the canonical homomorphism $\varphi_{\ast}:H_n(X)\to \check{H}_n(X)$ is an isomorphism, see also \cite{EKSingularSurj} regarding the surjectivity of $\varphi_{\ast}$. Examples show that $\Psi_{n}$ can fail to be an isomorphism if $X$ is semilocally $\pi_n$-trivial but not $LC^{n-1}$ (Example \ref{examplenotuvnminusone}) or if $X$ is $LC^{n-1}$ but not semilocally $\pi_n$-trivial (Example \ref{examplenotsemilocallypintrivial}).

The authors are grateful to the referee for many suggestions, which substantially improved the exposition of this paper.

\section{Preliminaries and Notation}\label{section2prelims}

Throughout this paper, $X$ is assumed to be a path-connected topological
space with basepoint $x_0$. The unit interval is denoted $\ui$ and $S^n$ is the unit $n$-sphere with basepoint $d_0=(1,0,\dots ,0)$. The $n$-th homotopy group of $(X,x_0)$ is denoted $\pi_n(X,x_0)$. If $f:(X,x_0)\to (Y,y_0)$ is a based map, then $f_{\#}:\pi_n(X,x_0)\to \pi_n(Y,y_0)$ is the induced homomorphism.

A \textit{path} in a space $X$ is a map $\alpha:\ui\rightarrow
X $ from the unit interval. The \textit{reverse} of $\alpha$ is the path
given by $\alpha^{-}(t)=\alpha(1-t)$ and the concatenation of two paths $\alpha,\beta$ with $\alpha(1)=\beta(0)$ is denoted $\alpha\cdot \beta$. Similarly, if $f,g:S^n\to X$ are maps based at $x\in X$, then $f\cdot g$ denotes the usual $n$-loop concatenation and $f^{-}$ denotes the reverse map. We may write $\prod_{i=1}^{m}f_i$ to denote an $m$-fold concatenation $f_1\cdot f_2\cdot \,\cdots\, \cdot f_m$.

\subsection{Simplicial complexes }

We make heavy use of standard notation and theory of abstract and geometric simplicial complexes, which can be found in texts such as \cite{MS82} and \cite{Mu84}. We briefly recall relevant notation.

For an abstract (geometric) simplicial complex $K$ and integer $r\geq 0$, $K_r$ denotes the $r$-skeleton of $K$. If $K$ is abstract, $|K|$ denotes the geometric
realization of $K$ with the weak topology. If $K$ is geometric, then $\sd^{m}K$ denotes the $m$-th
barycentric subdivision of $K$ and if $v$ is a vertex of $K$, then $\st(v,K)$ denotes the open star of the vertex $v$. When $L\subseteq K$ is a subcomplex, $\sd^{m}L$ is a subcomplex of $\sd^{m}K$.  If $\sigma=\{v_0,v_1,\dots,v_r\}$ is a $r$-simplex of $K$, then $
[v_0,v_1,...,v_r]$ denotes the $r$-simplex of $|K|$ with the indicated orientation.

We frequently make use of the standard $n$-simplex $\Delta_{n}$ in $\bbr^{n}$ spanned by the origin ${\bf o}$ and standard unit vectors. Since the boundary $\partial\Delta_{n}=(\Delta_{n})_{n-1}$ is homeomorphic to $S^{n-1}$, we fix a based homeomorphism $\partial \Delta_{n}\cong S^{n-1}$ that allows us to represent elements of $\pi_n(X,x_0)$ by maps $(\partial \Delta_{n+1},{\bf o})\to (X,x_0)$.

\subsection{The \v{C}ech expansion and shape homotopy groups}

We now recall the construction of the first shape homotopy group $\check{\pi}_{1}(X,x_0)$ via the \v{C}ech expansion. For more details, see \cite{MS82}.

Let $\mathcal{O}(X)$ be the set of open covers of $X$ directed by refinement; we write $\scrv\succeq \scru$ when $\scrv$ refines $\scru$. Similarly, let $\mathcal{O}(X,x_0)$ be the set of open covers with a
distinguished element containing $x_0$, i.e. the set of pairs $(\mathscr{U},U_0)$ where $\mathscr{U}\in \mathcal{O}(X)$, $U_0\in \mathscr{U}$, and $x_0\in U_0$. We say $(\mathscr{V},V_0)$ refines $(\mathscr{U},U_0)$ if $\scrv\succeq \scru$ and $V_0\subseteq U_0$.

The nerve of a cover $(\mathscr{U},U_0)\in\mathcal{O}(X,x_0)$ is the
abstract simplicial complex $N(\mathscr{U})$ whose vertex set is $N(
\mathscr{U})_0=\mathscr{U}$ and vertices $A_0,...,A_n\in \mathscr{U}$ span
an n-simplex if $\bigcap_{i=0}^{n}A_i\neq \emptyset$. The vertex $U_0$ is
taken to be the basepoint of the geometric realization $|N(\mathscr{U})|$.
Whenever $(\mathscr{V},V_0)$ refines $(\mathscr{U},U_0)$, we can construct a
simplicial map $p_{\mathscr{U}\mathscr{V}}:N(\mathscr{V})\to N(\mathscr{U})$
, called a \textit{projection}, given by sending a vertex $V\in N(\mathscr{V}
)$ to a vertex $U\in \mathscr{U}$ such that $V\subseteq U$. In particular, we make a convention that $p_{\mathscr{U}\mathscr{V}}(V_0)=U_0$. Any such assignment of vertices extends linearly to a simplicial map. Moreover, the induced map $|p_{\mathscr{U}\mathscr{V}}|:|N(
\mathscr{V})|\to|N(\mathscr{U})|$ is unique up to based homotopy. Thus the
homomorphism $p_{\mathscr{U}\mathscr{V}\#}:\pi_{1}(|N(\mathscr{V})|,V_0)\to \pi_{1}(|N(\mathscr{U})|,U_0)$ induced on fundamental groups is (up to coherent isomorphism) independent of the choice of simplicial map.

Recall that an open cover $\scru$ of $X$ is normal if it admits a partition of unity subordinated to $\scru$. Let $\Lambda$ be the subset of $\mathcal{O}(X,x_0)$ (also directed by refinement) consisting of pairs $(\mathscr{U},U_0)$ where $\scru$ is a normal open cover of $X$ and such that there is a partition of unity $\{\phi_{U}\}_{U\in \mathscr{U}}$ subordinated to $\scru$ with $\phi_{U_0}(x_0)=1$. It is well-known that every open cover of a paracompact Hausdorff space $X$ is normal. Moreover, if $(\scru,U_0)\in \mathcal{O}(X,x_0)$, it is easy to refine $(\scru,U_0)$ to a cover $(\scrv,V_0)$ such that $V_0$ is the only element of $\scrv$ containing $x_0$ and therefore $(\scrv,V_0)\in \Lambda$. Thus, for paracompact Hausdorff $X$, $\Lambda$ is cofinal in $\mathcal{O}(X,x_0)$.

The \textit{$n$-th shape homotopy group} is the inverse limit 
\begin{equation*}
\check{\pi}_{n}(X,x_0)=\varprojlim\left(\pi_{n}(|N(\mathscr{U})|,U_0),p_{
\mathscr{U}\mathscr{V}\#},\Lambda\right).
\end{equation*}
This group is also referred to as the $n$-th \v{C}ech homotopy group.

Given an open cover $(\mathscr{U},U_0)\in \mathcal{O}(X,x_0)$, a map $p_{\mathscr{U}}:X\to |N(\mathscr{U})|$ is a \textit{(based) canonical map} if $p_{\mathscr{U}}^{-1}(\st(U,N(\mathscr{U})))\subseteq U$ for each $U\in \mathscr{U}$ and $p_{\scru}(x_0)=U_0$. Such a canonical map is guaranteed to exist if $(\mathscr{U},U_0)\in
\Lambda $: find a locally finite partition of unity $\{\phi_{U}\}_{U\in \mathscr{U}}$ subordinated to $\scru$ such that $\phi_{U_0}(x_0)=1$. When $U\in \mathscr{U}$ and $x\in U$, determine $p_{\mathscr{U}}(x)$ by requiring its barycentric coordinate belonging to the vertex $U$ of $|N(\mathscr{U})|$ to be $\phi_{U}(x)$. According to this construction, the requirement $\phi_{U_0}(x_0)=1$ gives $p_{\scru}(x_0)=U_0$.

A canonical map $p_{\scru}$ is unique up to based homotopy and whenever $(\mathscr{V},V_0)$ refines $(\mathscr{U},U_0)$; the compositions $p_{\mathscr{U}\mathscr{V}}\circ p_{\mathscr{V}}$ and $p_{\mathscr{U}}$ are homotopic as based maps. Hence, for $n\geq 1$, the homomorphisms $p_{\mathscr{U}\#}:\pi_{n}(X,x_0)\to \pi_{n}(|N(\mathscr{U})|,U_0)$ satisfy $p_{\mathscr{U}\mathscr{V}\#}\circ p_{\mathscr{V}\#}=p_{\mathscr{U}\#}$. These homomorphisms
induce the following canonical homomorphism to the limit, which is natural in the continuous maps of based spaces:
\begin{equation*}
\Psi_{n}:\pi_{n}(X,x_0)\to \check{\pi}_{n}(X,x_0)\text{ given by }\Psi_{n}([f])=\left([p_{\mathscr{U}}\circ f]\right).
\end{equation*}

The subgroup $\ker(\Psi_{n})$, which we refer to as the \textit{$n$-th shape kernel} is, in a rough sense, an algebraic measure of the $n$-dimensional homotopical information lost when approximating $X$ by polyhedra. Since $(p_{\scru})$ forms an $\mathbf{HPol}$-expansion of $X$ \cite[App. 1 \S 3.2 Theorem 8]{MS82}, we have $[f]\in \pi_n(X,x_0)\backslash \ker(\Psi_n)$ if and only if there exist a polyhedron $K$ and a map $p:(X,x_0)\to (K,k_0)$ such that $p_{\#}([f])\neq 0$ in $\pi_n(K,k_0)$. Of utmost importance is the situation when $\ker(\Psi_{n})=0$. In this case, $\pi_n(X,x_0)$ can be understood as a subgroup of $\check{\pi}_{n}(X,x_0)$, that is, the $n$-th shape group retains all the data in the $n$-th homotopy group of $X$. A space for which $\ker(\Psi_{n})=0$ is said to be $\pi_n$-\textit{shape injective}.

\section{Higher Spanier Groups}\label{section3spaniergroups}

To define higher Spanier groups as in \cite{BKPnspanier}, we briefly recall the action of the fundamental groupoid on the higher homotopy groups of a space. Fix a retraction $R:S^n\times \ui\to S^n\times \{0\}\cup \{d_0\}\times \ui$. Given a map $f:(S^n,d_0)\to (X,y_0)$ and a path $\alpha:\ui\to X$ with $\alpha(0)=x_0$ and $\alpha(1)=y_0$, define $F:S^n\times \{0\}\cup \{d_0\}\times \ui\to X$ so that $g(x,0)=f(x)$ and $f(d_0,t)=\alpha(1-t)$. The \textit{path-conjugate of $f$ by} $\alpha$ is the map $\alpha\ast f:(S^n,d_0)\to (X,x_0)$ given by $\alpha\ast f(x)=F( R(x,1))$.

Path-conjugation defines the basepoint-change isomorphism $\varphi_{\alpha}:\pi_n(X,y_0)\to \pi_n(X,x_0)$, $\varphi_{\alpha}([f])=[\alpha\ast f]$. In particular, $[\alpha\ast f][\alpha\ast g]=[\alpha\ast (f\cdot g)]$. Additionally, if $[\alpha]=[\beta]$, which we write to mean that the paths $\alpha$ and $\beta$ are homotopic relative to their endpoints, then $[\alpha\ast f]=[\beta\ast f]$. Note that when $n=1$, $f:S^1\to X$ is a loop and $\alpha\ast f\simeq \alpha\cdot f\cdot \alpha^{-}$.

\begin{definition}
Let $n\geq 1$ and $\alpha:(\ui,0)\to (X,x_0)$ be a path and $U$ be an open neighborhood of $\alpha(1)$ in $X$. Define \[[\alpha]\ast \pi_n(U)=\{[\alpha\ast f]\in \pi_n(X,x_0)\mid f(S^n)\subseteq U,f(d_0)=\alpha(1)\}.\]
\end{definition}

Since $[\alpha\ast f][\alpha\ast g]=[\alpha\ast (f\cdot g)]$, the set $[\alpha]\ast \pi_n(U)$ is a subgroup of $\pi_n(X,x_0)$.

\begin{definition}
Let $n\geq 1$, $\scru$ be an open cover of $X$, and $x_0\in X$. The \textit{$n$-th Spanier group of $(X,x_0)$ with respect to $\scru$} is the subgroup $\span(\scru,x_0)$ of $\pi_n(X,x_0)$ generated by the subgroups $[\alpha]\ast \pi_n(U)$ for all pairs $(\alpha,U)$ with $\alpha(1)\in U$ and $U\in\scru$. In short:
\[\span(\scru,x_0)=\langle [\alpha]\ast \pi_n(U)\mid U\in \scru,\alpha(1)\in U\rangle\]
The \textit{$n$-th Spanier group of $(X,x_0)$} is the intersection 
\[\pi_{n}^{Sp}(X,x_0)=\bigcap_{\scru\in O(X)}\pi_{n}^{Sp}(\scru,x_0).\]We may refer to subgroups of the form $\span(\scru,x_0)$ as \textit{relative} Spanier groups and to $\pi_{n}^{Sp}(X,x_0)$ as the \textit{absolute} Spanier group.
\end{definition}

\begin{remark}
We note that our definition of $n$-th Spanier group is the ``unbased" definition from \cite{BKPnspanier}; see also \cite{FRVZ11} for more on ``based" Spanier groups, which is defined using covers of $X$ by \textit{pointed} open sets. The two notions agree for locally path connected spaces. When $n=1$, Spanier groups (absolute and relative to a cover) are normal subgroups of $\pi_1(X,x_0)$. In the case $n=1$, Spanier groups have been studied heavily due to their relationship to covering space theory \cite{Spanier66}. 
\end{remark}

\begin{remark}[Functorality]
Let $\bspaces$ denote the category of based topological spaces and based continuous functions and $\grp$ and $\ab$ denote the usual categories of groups and abelian groups respectively. If $f:(X,x_0)\to (Y,y_0)$ is a map and $\scrv$ is an open cover of $Y$, then $\scru=\{f^{-1}(V)\mid V\in\scrv\}$ is an open cover of $X$ such that $f_{\#}(\pi_n(\scru,x_0))\subseteq \pi_n(\scrv,y_0)$. It follows that $f_{\#}(\span(X,x_0))\subseteq \span(Y,y_0)$. Thus $(f_{\#})|_{\span(X,x_0)}:\span(X,x_0)\to \span(Y,y_0)$ is well-defined showing that $\pi_{1}^{Sp}:\bspaces\to \grp$ and $\span:\bspaces\to \ab$, $n\geq 2$, are functors \cite[Theorem 4.2]{BKPnspanier}. Moreover, if $g:(Y,y_0)\to (X,x_0)$ is a based homotopy inverse of $f$, then $(f_{\#})|_{\span(X,x_0)}$ and $(g_{\#})|_{\span(Y,y_0)}$ are inverse isomorphisms. Hence, these functors descend to functors $\mathbf{hTop_{\ast}}\to\grp$ and $\mathbf{hTop_{\ast}}\to\ab$ where $\mathbf{hTop_{\ast}}$ is the category of based spaces and basepoint-relative homotopy classes of based maps.
\end{remark}

\begin{remark}[Basepoint invariance]\label{basepointinvarianceremark}
Suppose $x_0,x_1\in X$ and $\beta:\ui\to X$ is a path from $x_1$ to $x_0$, and $\varphi_{\beta}:\pi_n(X,x_0)\to \pi_n(X,x_1)$, $\varphi_{\beta}([g])=[\beta\ast g]$ is the basepoint-change isomorphism. If $[\alpha\ast f]$ is a generator of $\span(\scru,x_0)$, then $\varphi_{\beta}([\alpha\ast f])=[(\beta\cdot\alpha)\ast f]$ is a generator of $\span(\scru,x_1)$. It follows that $\varphi_{\beta}(\span(\scru,x_0))=\span(\scru,x_1)$. Moreover, in the absolute case, we have $\varphi_{\beta}(\span(X,x_0))=\span(X,x_1)$. In particular, changing the basepoint of $X$ does not change the isomorphism type of the $n$-th Spanier group, particularly its triviality.
\end{remark}

In terms of our choice of generators, a generic element of $\span(\scru,x_0)$ is a product $\prod_{i=1}^{m}[\alpha_i\ast f_i]$ where each map $f_i:S^n\to X$ has an image in some open set $U_i\in\scru$ (see Figure \ref{figgenerator}). The next lemma identifies how such products might actually appear in practice and motivates the proof of our key technical lemma, Lemma \ref{secondinclusionlemma}. Recall that $(\sd^m\Delta_{n+1})_{n}$ is the union of the boundaries of the $(n+1)$-simplices in the $m$-th barycentric subdivision $\sd^m\Delta_{n+1}$.

\begin{figure}[H]
\centering \includegraphics[height=2.3in]{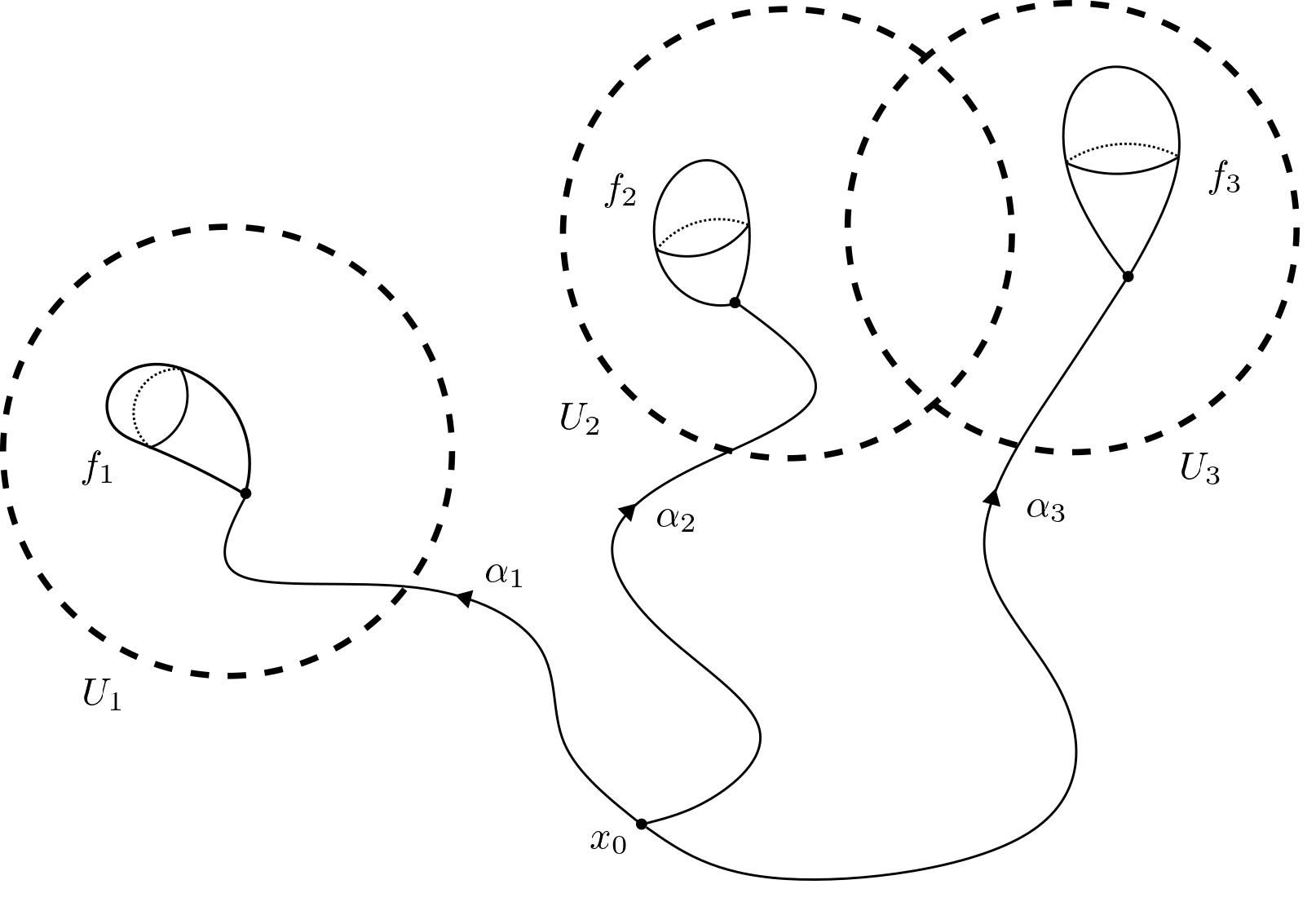}
\caption{\label{figgenerator}An element of $\pi_{2}^{Sp}(\scru,x_0)$, which is a product of three path-conjugate generators $[\alpha_i\ast f_i]$. }
\end{figure}

\begin{lemma}\label{emptyroomlemma}
For $m,n\in\bbn$, let $\scru$ be an open cover of $X$ and $f:((\sd^m\Delta_{n+1})_{n},{\bf o})\to (X,x_0)$ be a map such that for every $(n+1)$-simplex $\sigma$ of $\sd^m\Delta_{n+1}$, we have $f(\partial \sigma)\subseteq U$ for some $U\in\scru$. Then $f_{\#}(\pi_n((\sd^m\Delta_{n+1})_{n},{\bf o}))\subseteq  \pi_{n}^{Sp}(\scru,x_0)$.
\end{lemma}

\begin{proof}
The case $n=1$ is proved in \cite{BFThickSpan}. Suppose $n\geq 2$ and set $K=\sd^m\Delta_{n+1}$. The set $\scrw=\{f^{-1}(U)\mid U\in\scru\}$ is an open cover of $K_{n}=(\sd^m\Delta_{n+1})_n$ such that $f_{\#}(\pi_{n}^{Sp}(\scrw,{\bf o}))\subseteq \pi_{n}^{Sp}(\scru,x_0)$ and for every $(n+1)$-simplex $\sigma$ in $K$, we have $\partial \sigma\subseteq f^{-1}(U)$ for some $U\in\scru$. Thus it suffices to prove $\pi_n(K_{n},{\bf o})\subset \span(\scrw,{\bf o})$. Let $S$ be the set of $(n+1)$-simplices of $K$. Since $n\geq 2$, $K_{n}$ is simply connected. Standard simplicial homology arguments give that the reduced singular homology groups of $K_{n}$ are trivial in dimension $<n$ and $H_{n}(K_{n})$ is a finitely generated free abelian group. A set of free generators for $H_n(K_{n})$ can be chosen by fixing the homology class of a simplicial map $g_{\sigma}:\partial \Delta_{n+1}\to K_n$ that sends $\partial \Delta_{n+1}$ homeomorphically onto the boundary of an $(n+1)$-simplex $\sigma \in S$. Thus $K_n$ is $(n-1)$-connected and the Hurewicz homomorphism $h:\pi_k(K_n,{\bf o})\to H_k(K_n)$ is an isomorphism for all $1\leq k\leq n$. In particular, let $p_{\sigma}:I\to K_n$ be any path from ${\bf o}$ to $g_{\sigma}({\bf o})$. Then $\pi_n(K_n,{\bf o})$ is freely generated by the path-conjugates $[p_{\sigma}\ast g_{\sigma}]$, $\sigma\in S$. By assumption, for every $\sigma\in S$, $[p_{\sigma}\ast g_{\sigma}]$ is a generator of $\pi_{n}^{Sp}(\scrw,{\bf o})$. Since $\pi_{n}^{Sp}(\scrw,{\bf o})$ contains all the generators of $\pi_n(K_n,{\bf o})$, the inclusion $\pi_n(K_{n},{\bf o})\subset\span(\scrw,{\bf o})$ follows.
\end{proof}

To characterize the triviality of relative Spanier groups, we establish the following terminology.

\begin{definition}\label{slscdef}
Let $n\geq 0$ and $x\in X$. We say the space $X$ is 
\begin{enumerate}
\item \textit{semilocally $\pi_n$-trivial at} $x$ if there exists an open neighborhood $U$ of $x$ in $X$ such that every map $S^n\to U$ is null-homotopic in $X$.
\item \textit{semilocally $n$-connected at} $x$ if there exists an open neighborhood $U$ of $x$ in $X$ such that every map $S^k\to X$, $0\leq k\leq n$ is null-homotopic in $X$.
\end{enumerate}
We say $X$ is \textit{semilocally $\pi_n$-trivial} (resp. semilocally $n$-connected) if it has this property at all of its points.
\end{definition}

It is straightforward to see that $X$ is semilocally $n$-connected at $x$ if and only if $X$ is semilocally $\pi_k$-trivial at $x$ for all $0\leq k\leq n$.

\begin{remark}
A space $X$ is semilocally $\pi_n$-trivial if and only if $X$ admits an open cover $\scru$ such that $\span(\scru,x_0)$ is trivial \cite[Theorem 3.7]{BKPnspanier}. Moreover, $X$ is semilocally $n$-connected if and only if $X$ admits an open cover $\scru$ such that $\pi_{k}^{Sp}(\scru,x_0)$ is trivial for all $1\leq k\leq n$. Note that local path connectivity is independent of the properties given in Definition \ref{slscdef}.
\end{remark}

Attempting a proof of Theorem \ref{maintheorem}, one should not expect the groups $\span(\scru,x_0)$ and $\ker(p_{\scru\#})$ to agree ``on the nose." Indeed, the following example shows that we should not expect the equality $\span(\scru,x_0)=\ker(p_{\scru\#})$ to hold even in the ``nicest" local circumstances. 

\begin{example}
Let $X=S^2\vee S^2$ and $W$ be a contractible neighborhood of $d_0$ in $S^2$. Set $U_1=S^2\vee W$ and $U_2=W\vee S^2$ and consider the open cover $\scru=\{U_1,U_2\}$ of $X$. Then $\pi_{3}^{Sp}(\scru,x_0)\cong \bbz^2$ is freely generated by the homotopy classes of the two inclusions $i_1,i_2:S^2\to X$. However, $\pi_{3}(X)\cong \bbz^{3}$ is freely generated by $[i_1]$, $[i_2]$, and the Whitehead product $\llbracket i_1,i_2\rrbracket$. However $|N(\scru)|$ is a $1$-simplex and is therefore contractible. Thus $\ker (p_{\scru\#})$ is equal to $\pi_{3}(X)$ and contains $\llbracket i_1,i_2\rrbracket$. Even though the spaces $X,U_1,U_2$ are locally contractible and the elements of $\scru$ are $1$-connected, $\span(\scru,x_0)$ is a proper subgroup of $\ker(p_{\scru\#})$. One can view this failure as the result of two facts: (1) The sets $U_i$ are not $2$-connected and (2) the definition of Spanier group does not allow one to generate homotopy classes by taking Whitehead products of maps $S^2\to U_i$ in the neighboring elements of $\scru$.
\end{example}

First, we show the inclusion $\span(X,x_0)\subseteq \ker (\Psi_{n})$ holds in full generality. Recall the intersections $\span(X,x_0)=\bigcap_{\scru\in O(X)}\span(\scru,x_0)$ and $\ker(\Psi_{n})=\bigcap_{(\scru,U_0)\in\Lambda}\ker(p_{\scru\#})$ are formally indexed by different sets.

\begin{lemma}\label{inclusionlemma1}
For every open cover $\scru$ of $X$ and canonical map $p_{\scru}:X\to|N(\scru)|$, there exists a refinement $\scrv \succeq\scru$ such that $\span(\scrv,x_0)\subseteq \ker(p_{\scru\#})$ in $\pi_n(X,x_0)$.
\end{lemma}

\begin{proof}
Let $\scru\in O(X)$. The stars $\st(U,|N(\scru)|)$, $U\in\scru$ form an open cover of $|N(\scru)|$ by contractible sets and therefore $\scrv=\{p_{\scru}^{-1}(\st(U,|N(\scru)|))\mid U\in\scru\}$ is an open cover of $X$. Since $p_{\scru}$ is a canonical map, we have $p_{\scru}^{-1}(\st(U,|N(\scru)|))\subseteq U$ for all $U\in\scru$. Thus $\scrv$ is a refinement of $\scru$. A generator of $\span(\scrv,x_0)$ is of the form $[\alpha\ast f]$ for a map $f:S^n\to p_{\scru}^{-1}(\st(U,|N(\scru)|))$. However, $p_{\scru}\circ f$ has image in the contractible open set $\st(U,|N(\scru)|)$ and is therefore null-homotopic. Thus $p_{\scru\#}([\alpha\ast f])=0$. We conclude that $p_{\scru\#}(\span(\scrv,x_0))=0$.
\end{proof}

\begin{corollary}\label{firstinclusioncor}\cite[Theorem 4.14]{BKPnspanier}
Let $n\geq 1$. For any based space $(X,x_0)$, we have $\span(X,x_0)\subseteq \ker(\Psi_{n})$.
\end{corollary}

\begin{proof}
Suppose $[f]\in \span(X,x_0)$. Given a normal, based open cover $(\scru,U_0)\in \Lambda$ and any canonical map $p_{\scru}:X\to |N(\scru)|$, Lemma \ref{inclusionlemma1} ensures we can find a refinement $\scrv\succeq \scru$ such that $\span(\scrv,x_0)\subseteq \ker(p_{\scru\#})$. Thus $[f]\in \span(\scrv,x_0)\subseteq \ker(p_{\scru\#})$. Since $(\scru,U_0)$ is arbitrary, we conclude that $[f]\in \ker(\Psi_{n})$.
\end{proof}

\begin{example}[higher earring spaces]
An important space, which we will call upon repeatedly for examples, is the \textit{$n$-dimensional earring space} \[\bbe_n= \bigcup_{j\in\bbn}\left\{\bfx\in\bbr^{n+1}\mid \|\bfx-(1/j,0,0,\dots,0)\|=1/j\right\},\] which is a shrinking wedge (one-point union) of $n$-spheres with basepoint the origin ${\bf o}$. It is known that $\bbe_n$ is $(n-1)$-connected, locally $(n-1)$-connected, and $\pi_n$-shape injective for all $n\geq 1$ \cite{EK00higher,MM}. However, $\bbe_n$ is not semilocally $\pi_n$-trivial. Thus $\span(\scru,{\bf o})\neq 0$ for any open cover $\scru$ of $\bbe_n$ even though in the absolute case $\span(\bbe_n,{\bf o})$ is trivial.
\end{example}

\begin{example}
Let $n\geq 3$ and notice that $\bbe_1\vee \bbe_{n}$ is not semilocally $\pi_1$-trivial (since it has $\bbe_1$ as a retract) and therefore fails to be semilocally $(n-1)$-connected. However, it has recently been shown that $\pi_{k}(\bbe_1\vee \bbe_{n})=0$ for $2\leq k\leq n-1$ and that $\bbe_1\vee\bbe_{n}$ is $\pi_{n}$-shape injective \cite{BrazSequentialnconn}. Thus $\bbe_1\vee \bbe_{n}$ is semilocally $\pi_k$-trivial for all $k\leq n-1$ except $k=1$ and $\pi_{n}^{Sp}(\bbe_1\vee \bbe_{n},{\bf o})=0$. Thus the failure to be semilocally $n$-connected can occur at single dimension less than $n$.
\end{example}

\section{Recursive Extension Lemmas}\label{section4extensions}

Toward a proof of the inclusion $\ker(\Psi_n)\subseteq \span(X,x_0)$ for $LC^{n-1}$ space $X$, we introduce some convenient notation and definitions. If $\scru$ is an open cover and $A\subseteq X$, then $\St(A,\scru)=\bigcup\{U\in\scru\mid A\cap U\neq\emptyset\}$. Note that if $A\subseteq B$, then $\St(A,\scru)\subseteq \St(B,\scru)$. Also if $\scrv\succeq \scru$, then $\St(A,\scrv)\subseteq \St(A,\scru)$. We take the following terminology from \cite{Willard}.

\begin{definition}
Let $\scru,\scrv\in O(X)$. 
\begin{enumerate}
\item We say $\scrv$ is a \textit{barycentric-star refinement} of $\scru$ if for every $x\in X$, we have $\St(x,\scrv)\subseteq U$ for some $U\in\scru$. We write $\scrv\succeq_{\ast}\scru$.
\item We say $\scrv$ is a \textit{star refinement} of $\scru$ if for every $V\in\scrv$, we have $\St(V,\scrv)\subseteq U$ for some $U\in\scru$. We write $\scrv\succeq_{\ast\ast}\scru$.
\end{enumerate}
\end{definition}

Note that if $\scru\preceq_{\ast}\scrv\preceq_{\ast}\scrw$, then $\scru\preceq_{\ast\ast}\scrw$.

\begin{lemma}\cite{StoneParacompact}\label{paracompactlemma}
A $T_1$ space $X$ is paracompact if and only if for every $\scru\in O(X)$ there exists $\scrv\in O(X)$ such that $\scrv\succeq_{\ast} \scru$.
\end{lemma}

\begin{definition}\cite[Ch. I \S 3.2.5]{MS82}
Let $n\in \{0,1,2,3,\dots,\infty\}$. A space $X$ is $LC^n$ at $x\in X$ if for every neighborhood $U$ of $x$, there exists a neighborhood $V$ of $x$ such that $V\subseteq U$ and such that for all $0\leq k\leq n$ ($k<\infty$ if $n=\infty$), every map $f:\partial \Delta_{k+1}\to V$ extends to a map $g:\Delta_{k+1}\to U$. We say $X$ is $LC^n$ if $X$ is $LC^n$ at all of its points.
\end{definition}

We have the following evident implications for both the point-wise and global properties:
\[X\text{ is loc. }n\text{-connected}\,\, \Rightarrow\,\, X\text{ is }LC^n\,\,\Rightarrow\,\, X\text{ is semiloc. }n\text{-connected}\]
For first countable spaces, the $LC^{n}$ property is equivalent to the ``$n$-tame" property in \cite{BrazSequentialnconn} defined in terms of shrinking sequences of maps.

\begin{definition}
Suppose $\scrv\succeq \scru$ in $O(X)$.
\begin{enumerate}
\item We say $\scrv$ is an \textit{$n$-refinement of }$\scru$, and write $\scrv\succeq^{n} \scru$, if for all $0\leq k\leq n$, $V\in\scrv$, and maps $f:\partial \Delta_{k+1}\to V$, there exists $U\in\scru$ with $V\subseteq U$ and a continuous extension $g:\Delta_{k+1}\to U$ of $f$. 
\item We say $\scrv$ is an \textit{$n$-barycentric-star refinement of }$\scru$, and write $\scrv\succeq_{\ast}^{n} \scru$, if for every $0\leq k\leq n$, for every $x\in X$, and every map $f:\partial \Delta_{k+1}\to \St(x,\scrv)$, there exists $U\in\scru$ with $\St(x,\scrv)\subseteq U$ and a continuous extension $g:\Delta_{k+1}\to U$ of $f$.
\end{enumerate}
\end{definition}

Note that if $\scrv\succeq^{n}\scru$ (resp. $\scrv\succeq_{\ast}^{n}\scru$), then $\scrv\succeq^{k}\scru$ (resp. $\scrv\succeq_{\ast}^{k}\scru$) for all $0\leq k\leq n$.

\begin{lemma}\label{refinementlemma2}
Suppose $X$ is paracompact, Hausdorff, and $LC^n$. For every $\scru\in O(X)$, there exists $\scrv\in O(X)$ such that $\scrv\succeq_{\ast}^{n} \scru$.
\end{lemma}

\begin{proof}
Let $\scru\in O(X)$. Since $X$ is $LC^n$, for every $U\in\scru$ and $x\in U$, there exists an open neighborhood $W(U,x)$ such that $W(U,x)\subseteq U$ and such that for all $0\leq k\leq n$, each map $f:\partial \Delta_{k+1}\to W(U,x)$ extends to a map $g:\Delta_{k+1}\to U$. Let $\scrw=\{W(U,x)\mid U\in\scru, x\in U\}$ and note $\scrw\succeq^{n} \scru$. Since $X$ is paracompact Hausdorff, by Lemma \ref{paracompactlemma}, there exists $\scrv\in O(X)$ such that $\scrv\succeq_{\ast}\scrw$. 

Fix $x'\in X$. Then $\St(x',\scrv)\subseteq W(U,x)$ for some $x\in U\in\scru$. Then $\St(x',\scrv)\subseteq U$. Moreover, if $0\leq k\leq n$ and $f:\partial \Delta_{k+1}\to \St(x',\scrv)$ is a map, then since $f$ has image in $W(U,x)$, there is an extension $g:\Delta_{k+1}\to U$. This verifies that $\scrv\succeq_{\ast}^{n} \scru$.
\end{proof}

For the next two lemmas, we fix $n\in\bbn$, a geometric simplicial complex $K$ with $\dim K=n+1$, and a subcomplex $L\subseteq K$ with $\dim L\leq n$. Let $M[k]=L\cup K_k$ denote the union of $L$ and the $k$-skeleton of $K$. Since $L\subseteq K_n$, $M[n]=K_n$ is the union of the boundaries of the $(n+1)$-simplices of $K$. Later we will consider the cases where (1) $K=\sd^m\Delta_{n+1}$ and $L=\sd^m\partial \Delta_{n+1}$ and (2) $K=\sd^m\partial \Delta_{n+2}$ and $L=\{{\bf o}\}$.

\begin{lemma}[Recursive Extensions]\label{recursionlemma}
Suppose $1\leq k\leq n$, $\scru\preceq_{\ast}\scrv\preceq_{\ast}^{k-1} \scrw$, $m\in\bbn$, and $f:M[k-1]\to X$ is a map such that for every $(n+1)$-simplex $\sigma$ of $K$, we have $f(\sigma\cap M[k-1])\subseteq W_{\sigma}$ for some $W_{\sigma}\in\scrw$. Then there exists a continuous extension $g:M[k]\to X$ of $f$ such that for every $(n+1)$-simplex $\sigma$ of $K$, we have $g(\sigma\cap M[k])\subseteq U_{\sigma}$ for some $U_{\sigma}\in\scru$.
\end{lemma}

\begin{proof}
Supposing the hypothesis, we must extend $f$ to the $k$-simplices of $M[k]$ that do not lie in $L$. Let $\tau$ be a $k$-simplex of $M[k]$ that does not lie in $L$ and let $S_{\tau}$ be the set of $(n+1)$-simplices in $K$ that contain $\tau$. By assumption, $S_{\tau}$ is non-empty. We make some general observations first. Since $f$ maps the $(k-1)$-skeleton of each $(n+1)$-simplex $\sigma\in S_{\tau}$ into $W_{\sigma}$ and $\partial \tau$ lies in this $(k-1)$-skeleton, we have $f(\partial\tau)\subseteq \bigcap_{\sigma\in S_{\tau}}W_{\sigma}$. Thus, for all $\tau$, we have \[f(\partial \tau)\subseteq \bigcap_{\sigma\in S_{\tau}}\St(W_{\sigma},\scrv).\]

Fix a vertex $v_{\tau}$ of $\tau$ and let $x_{\tau}=f(v_{\tau})$. Then $x_{\tau}\in W_{\sigma}\subseteq \St(x_{\tau},\scrw)$ whenever $\sigma\in S_{\tau}$. Since $\scrw\succeq_{\ast}^{k-1}\scrv$, we may find $V_{\tau}\in \scrv$ such that $\St(x_{\tau},\scrw)\subseteq V_{\tau}$ and such that every map $\partial \Delta_{k}\to \St(x_{\tau},\scrw)$ extends to a map $\Delta_k\to V_{\tau}$. In particular, $f|_{\partial \tau}:\partial \tau\to W_{\sigma}$ extends to a map $\tau\to V_{\tau}$. We define $g:M[k]\to X$ so that it agrees with $f$ on $M[k-1]$ and so that the restriction of $g$ to $\tau$ is a choice of continuous extension $\tau\to V_{\tau}$ of $f|_{\partial \tau}$. 

We now choose the sets $U_{\sigma}$. Fix an $(n+1)$-simplex $\sigma$ of $K$. If the $k$-skeleton of $\sigma$ lies entirely in $L$, we choose any $U_{\sigma}\in\scru$ satisfying $W_{\sigma}\subseteq U_{\sigma}$. Suppose there exists at least one $k$-simplex in $\sigma$ not in $L$. Then whenever $\tau$ is a $k$-simplex of $\sigma$ not in $L$, we have $W_{\sigma}\subseteq \St(x_{\tau},\scrw)\subseteq  V_{\tau} $. Fix a point $y_{\sigma}\in W_{\sigma}$. The assumption that $\scrv\succeq_{\ast}\scru$ implies that there exists $U_{\sigma}\in\scru$ such that $\St(y_{\sigma},\scrv)\subseteq U_{\sigma}$. In this case, we have $W_{\sigma}\subseteq V_{\tau}\subseteq U_{\sigma}$ whenever $\tau$ is a $k$-simplex of $\sigma$ not in $L$.

Finally, we check that $g$ satisfies the desired property. Again, fix an $(n+1)$-simplex $\sigma$ of $K$. If $\tau$ is a $k$-simplex of $\sigma$ not in $L$, our definition of $g$ gives $g(\tau)\subseteq V_{\tau}\subseteq U_{\sigma}$. If $\tau '$ is a $k$-simplex in $\sigma\cap L$, then $g(\tau ')=f(\tau ')\subseteq W_{\sigma}\subseteq U_{\sigma}$. Overall, this shows that $g(\sigma \cap M[k])\subseteq U_{\sigma}$ for each $(n+1)$-simplex $\sigma$ of $K$.
\end{proof}

A direct, recursive application of the previous lemma is given in the following statement.

\begin{lemma}\label{appliedrecursionlemma}
Suppose there is a sequence of open covers 
\[\scrw_{n}\preceq_{\ast}\scrv_{n}\preceq_{\ast}^{n-1}\scrw_{n-1}\preceq_{\ast}\cdots \preceq_{\ast}^{2}\scrw_{2}\preceq _{\ast} \scrv_{2}\preceq_{\ast}^{1} \scrw_{1}\preceq_{\ast} \scrv_{1}\preceq_{\ast}^{0} \scrw_{0}\]
and a map $f_0:M[0]\to X$ such that for every $(n+1)$-simplex $\sigma$ of $K$, we have $f_0(\sigma\cap M[0])\subseteq W$ for some $W\in \scrw_{0}$. Then there exists an extension $f_{n}:M[n]\to X$ of $f_0$ such that for every $(n+1)$-simplex $\sigma$ of $K$, we have $f_{n}(\partial \sigma)\subseteq U$ for some $U\in \scrw_{n}$.
\end{lemma}

\section{A proof of Theorem \ref{maintheorem}}\label{section5mainresult}

We apply the extension results of the previous section in the case where $K=\sd^{m}\Delta_{n+1}$ for some $m\in\bbn$ and $L=\sd^{m}\partial \Delta_{n+1}$ so that $M[k]=L\cup K_k$ consists of the $n$-simplices of the boundary of $\Delta_{n+1}$ and the $k$-simplices of $\sd^m\Delta_{n+1}$ not in the boundary. Note that $M[n]$ is the union of the boundaries of the $(n+1)$-simplices of $\sd^m\Delta_{n+1}$.

\begin{lemma}\label{secondinclusionlemma}
Let $n\geq 1$. Suppose $X$ is paracompact, Hausdorff, and $LC^{n-1}$. Then for every open cover $\scru$ of $X$, there exists $(\scrv,V_0)\in \Lambda$ such that $\ker(p_{\scrv\#})\subseteq \span(\scru,x_0)$.
\end{lemma}

\begin{proof}
Suppose $\scru\in O(X)$. Since $X$ is paracompact, Hausdorff, and $LC^{n-1}$, we may apply Lemmas \ref{paracompactlemma} and \ref{refinementlemma2} to first find a sequence of refinements
:\[\scru=\scru_{n}\preceq_{\ast}\scrv_{n}\preceq_{\ast}^{n-1}\scru_{n-1}\preceq_{\ast}\cdots \preceq_{\ast}^{2}\scru_{2}\preceq _{\ast} \scrv_{2}\preceq_{\ast}^{1} \scru_{1}\preceq_{\ast} \scrv_{1}\preceq_{\ast}^{0} \scru_{0}\]
and then one last refinement $\scru_{0}\preceq_{\ast} \scrv_0=\scrv$. Let $V_0\in\scrv$ be any set containing $x_0$ and recall that since $X$ is paracompact Hausdorff $(\scrv,V_0)\in \Lambda$. We will show that $\ker (p_{\scrv\#})\subseteq \span(\scru,x_0)$. Note that $p_{\scrv}^{-1}(\st(V,N(\scrv)))\subseteq V$ by the definition of canonical map $p_{\scrv}$.

Suppose $[f]\in \ker (p_{\scrv\#})$ is represented by a map $f:(|\partial \Delta_{n+1}|,{\bf o})\to (X,x_0)$. We will show that $[f]\in\span(\scru,x_0)$. Then $p_{\scrv}\circ f:|\partial \Delta_{n+1}|\to |N(\scrv)|$ is null-homotopic and extends to a map $h:|\Delta_{n+1}|\to |N(\scrv)|$. Set $Y_V=h^{-1}(\st(V,N(\scrv)))$ so that $\scry=\{Y_V\mid V\in\scrv\}$ is an open cover of $|\Delta_{n+1}|$.

We find a particular simplicial approximation for $h$ using the cover $\scry$ \cite[Theorem 16.1]{Mu84}: let $\lambda$ be a Lebesgue number for $\scry$ so that any subset of $\Delta_{n+1}$ of diameter less than $\lambda$ lies in some element of $\scry$. Find $m\in\bbn$ such that each simplex in $\sd^{m}\Delta_{n+1}$ has diameter less than $\lambda/2$. Thus the star $\st(a,\sd^{m}\Delta_{n+1})$ of each vertex $a$ in $\sd^{m}\Delta_{n+1}$ lies in a set $Y_{V_a}\in\scry$ for some $V_a\in \scrv$. The assignment $a\mapsto V_a$ on vertices extends to a simplicial approximation $h':\sd^{m}\Delta_{n+1}\to N(\scrv)$ of $h$, i.e. a simplicial map $h'$ such that \[h(\st(a,\sd^{m}\Delta_{n+1}))\subseteq \st(h'(a),N(\scrv))=\st(V_a,N(\scrv))\] for each vertex $a$ \cite[Lemma 14.1]{Mu84}.

Let $K=\sd^{m}\Delta_{n+1}$ and $L=\sd^m\partial \Delta_{n+1}$ so that $M[k]=L\cup K_k$. First, we extend $f:L\to X$ to a map $f_0:M[0]\to X$. For each vertex $a$ in $K$, pick a point $f_0(a)\in V_a$. In particular, if $a\in L$, take $f_0(a)=f(a)$. This choice is well defined since, for a boundary vertex $a\in L$, we have
$p_{\scrv}\circ f(a)=  h(a)\in \st(V_a,|N(\scrv)|)$ and thus $f(a)\in p_{\scrv}^{-1}(\st(V_a,|N(\scrv|)))\subseteq V_a$.

Note that $h'$ maps every simplex $\sigma=[a_0,a_1,\dots,a_k]$ of $K$ to the simplex of $N(\scrv)$ spanned by $\{h'(a_i)\mid 0\leq i\leq k\}=\{V_{a_i}\mid 0\leq i\leq k\}$. By definition of the nerve, we have $\bigcap\{V_{a_i}\mid 0\leq i\leq k\}\neq \emptyset$. Pick a point $x_{\sigma}\in \bigcap\{V_{a_i}\mid 0\leq i\leq k\}$. 

By our initial choice of refinements, we have $\scru_0\preceq_{\ast} \scrv$. If $\sigma=[a_0,a_1,\dots,a_{n+1}]$ is an $(n+1)$-simplex of $K$, then $\St(x_{\sigma},\scrv)\subseteq U_{\sigma}$ for some $U_{\sigma}\in\scru$. In particular $\{f_0(a_i)\mid 0\leq i\leq n+1\}\subseteq \bigcup\{V_{a_i}\mid 0\leq i\leq n+1\}\subseteq U_{\sigma}$. Thus $f_0$ maps the $0$-skeleton of $\sigma$ into $U_{\sigma}$. If $1\leq k\leq n$, $\tau$ is a $k$-face of $\sigma\cap L$ with $a_i\in \tau$, then $p_{\scrv}\circ f_0(\int(\tau))=p_{\scrv}\circ f(\int(\tau))=h(\int(\tau))\subseteq h(\st(a_i,K))\subseteq \st(V_{a_i},|N(\scrv)|)$. It follows that \[f_0(\tau)\subseteq p_{\scrv}^{-1}(\st(V_{a_i},|N(\scrv)|))\subseteq V_{a_i}\subseteq U_{\sigma}.\]
Thus for every $n$-simplex in $\sigma\cap L$, we have $f_0(\tau)\subseteq U_{\sigma}$. We conclude that for every $(n+1)$-simplex $\sigma$ of $K$, we have $f_0(\sigma \cap M[0])\subseteq U_{\sigma}$.

By our choice of sequence of refinements, we are precisely in the situation to apply Lemma \ref{appliedrecursionlemma}. Doing so, we obtain an extension $f_{n}:M[n]\to X$ of $f$ such that for every $(n+1)$-simplex $\sigma$ of $K$, we have $f_{n}(\partial \sigma)\subseteq \mathbf{U}_{\sigma}$ for some $\mathbf{U}_{\sigma}\in\scru_{n}=\scru$. By Lemma \ref{emptyroomlemma}, we have $[f]=[f_{n}|_{\partial\Delta_{n+1}}]\in \pi_{n}^{Sp}(\scru,x_0)$.
\end{proof}

Finally, both inclusions have been established and provide a proof of our main result.

\begin{proof}[Proof of Theorem \ref{maintheorem}]

The inclusion $\pi_{n}^{Sp}(X,x_0)\subseteq \ker(\Psi_{n})$ holds in general by Corollary \ref{firstinclusioncor}. Under the given hypotheses, the inclusion $\ker(\Psi_{n})\subseteq \pi_{n}^{Sp}(X,x_0)$ follows from Lemma \ref{secondinclusionlemma}.
\end{proof}

When considering examples relevant to Theorem \ref{maintheorem}, it is helpful to compare $\pi_n$-shape injectivity with the following weaker property from \cite{GHnhomhausd}.

\begin{definition}
We say a space $X$ is \textit{$n$-homotopically Hausdorff at} $x\in X$ if no non-trivial element of $\pi_n(X,x)$ has a representing map in every neighborhood of $x$. We say $X$ is \textit{$n$-homotopically Hausdorff} if it is $n$-homotopically Hausdorff at each of its points.
\end{definition}

Clearly, $\pi_n$-shape injectivity $\Rightarrow$ $n$-homotopically Hausdorff. The next example, which highlights the effectiveness of Theorem \ref{maintheorem}, shows the converse is not true even for $LC^{n-1}$ Peano continua.

\begin{example}\label{mappingtorusexample}
Fix $n\geq 2$ and let $\ell_j:S^n\to \bbe_n$ be the inclusion of the $j$-th sphere and define $f:\bbe_n\to\bbe_n$ to be the shift map given by $f\circ \ell_j=\ell_{j+1}$. Let $M_f=\bbe_n\times [0,1]/\mathord{\sim}$, $(x,0)\mathord{\sim}(f(x),1)$ be the mapping torus of $f$. We identify $\bbe_n$ with the image of $\bbe_n\times\{0\}$ in $M_f$ and take ${\bf o}$ to be the basepoint of $M_f$. Let $\alpha:\ui\to M_f$ be the loop where $\alpha(t)$ is the image of $({\bf o},t)$. Then $M_f$ is locally contractible at all points other than those in the image of $\alpha$. Also, every point $\alpha(t)$ has a neighborhood that deformation retracts onto a homeomorphic copy of $\bbe_n$. Thus, since $\bbe_n$ is $LC^{n-1}$, so is $X$. It follows from Theorem \ref{maintheorem} that $\span(M_f,{\bf o})=\ker(\pi_n(M_f,{\bf o})\to \check{\pi}_n(M_f,{\bf o}))$. In particular, the Spanier group of $M_f$ contains all elements $[\alpha^k \ast g]$ where $g:S^n\to \bbe_n$ is a based map and $k\in\bbz$. Using the universal covering map $E\to M_f$ that ``unwinds" $\alpha$ and the relation $[g]=[\alpha\ast (f\circ g)]$ in $\pi_n(M_f,{\bf o})$, it is not hard to show that these are, in fact, the only elements of the $n$-th Spanier group. Hence, \[\ker(\pi_n(M_f,{\bf o})\to \check{\pi}_n(M_f,{\bf o}))=\{[\alpha^{k}\ast g]\mid [g]\in \pi_n(\bbe_n,{\bf o}),k\in\bbz\},\] which is an uncountable subgroup. Moreover, since $M_f$ is shape equivalent to the aspherical space $S^1$, we have $\check{\pi}_n(M_f,{\bf o})=0$ and thus $\pi_n(M_f,{\bf o})=\{[\alpha^{k}\ast g]\mid [g]\in \pi_n(\bbe_n,{\bf o}),k\in\bbz\}$.

It follows from this description that, even though $M_f$ is not $\pi_n$-shape injective, $M_f$ is $n$-homotopically Hausdorff. Indeed, it suffices to check this at the points $\alpha(t)$, $t\in\ui$. We give the argument for $\alpha(0)={\bf o}$, the other points are similar. If $0\neq h\in \pi_n(M_f,{\bf o})$ has a representative in every neighborhood of ${\bf o}$ in $M_f$, then clearly $h\in \ker(\Psi_n)$. Hence, $h=[\alpha^{k}\ast g]$ for $[g]\in \pi_n(\bbe_n,{\bf o})$ and $k\in\bbz$. Since $M_f$ retracts onto the circle parameterized by $\alpha$, the hypothesis on $h$ can only hold if $k=0$. However, there is a basis of neighborhoods of ${\bf o}$ in $M_f$ that deformation retract onto an open neighborhood of ${\bf o}$ in $\bbe_n$. Thus $[g]$ has a representative in every neighborhood of ${\bf o}$ in $\pi_n(\bbe_n,{\bf o})$, giving $h=[g]\in \ker(\pi_n(\bbe_n,{\bf o})\to \check{\pi}_n(\bbe_n,{\bf o}))=0$.
\end{example}

It is an important feature of Example \ref{mappingtorusexample} that $M_f$ is not simply connected and has multiple points at which it is not semilocally $\pi_n$-trivial. This motivates the following application of Theorem \ref{maintheorem}, which identifies a partial converse of the implication $\pi_n$-shape injective $\Rightarrow$ $n$-homotopically Hausdorff.

\begin{corollary}
Let $n\geq 2$ and $X$ be a simply-connected, $LC^{n-1}$, compact Hausdorff space such that $X$ fails to be semilocally $\pi_n$-trivial only at a single point $x\in X$. Then for every element $g\in \ker(\Psi_n)\subseteq \pi_n(X,x)$ and neighborhood $V$ of $x$, $g$ is represented by a map with image in $V$. In particular, if $X$ is $n$-homotopically Hausdorff at $x$, then $X$ is $\pi_n$-shape injective.
\end{corollary}

\begin{proof}
Let $0\neq g\in \ker(\Psi_{n})\subseteq \pi_n(X,x)$. By Theorem \ref{maintheorem}, $g\in \span(X,x)$. Since $X$ is compact Hausdorff, we may replace $O(X)$ by the cofinal sub-directed order $O_F(X)$ consisting of finite open covers $\scru$ of $X$ with the property that there is a unique $A_{\scru}\in \scru$ with $x\in A_{\scru}$. For each $\scru\in O_F(X)$, we can write $g=\prod_{i=1}^{m_{\scru}}[\alpha_{\scru,i}\ast f_{\scru,i}]$ where $f_{\scru,i}:S^n\to  U_{\scru,i}$ is a non-null-homotopic map for some $U_{\scru,i}\in \scru$ and $\alpha_{\scru,i}$ is a path from $x$ to $f_{\scru,i}(d_0)$. 

Let $V$ be an open neighborhood of $x$. We check that $g$ is represented by a map with image in $V$. Since $X$ is $LC^{0}$ at $x$, there exists an open neighborhood $V'$ of $x$ such that any two points of $V'$ may be connected by a path in $V$. Fix $\scru_0\in O_F(X)$ such that $A_{\scru_0}\subseteq V'$. Then $A_{\scrv}\subseteq V'$ whenever $\scrv\in O_F(X)$ refines $\scru_0$. 

We claim that for sufficiently refined $\scrv$, all of the maps $f_{\scrv,i}$ have image in $V'$. Suppose, to obtain a contradiction, there is a subset $T\subseteq \{\scrv\in O_F(X)\mid \scrv\succeq\scru_0\}$, which is cofinal in $O_F(X)$ and such that for every $\scrv\in T$ there exists $i_{\scrv}\in \{1,2,\dots ,m_{\scrv}\}$ and $d_{\scrv}\in S^n$ such that $f_{\scrv,i_{\scrv}}(d_{\scrv})\in U_{\scrv,i}\backslash V'\subseteq  U_{\scrv,i}\backslash A_{\scru_0}$. Since $X$ is compact, we may replace $\{f_{\scrv,i_{\scrv}}(d_{\scrv})\}$ with a subnet $\{x_j\}_{j\in J}$ that converges to a point $y\in X$. Here, $x_j=f_{\scrv_j,i_{\scrv_j}}(d_{\scrv_j})$ for some directed set $J$ and monotone, final function $J\to T$ given by $j\mapsto \scrv_j$. Let $Y$ be an open neighborhood of $y$ in $X$. Find $\scrw\in O_F(X)$ such that there exists $W_0\in\scrw$ such that $y\in W_0$ and $\St(W_0,\scrw)\subseteq Y$. Since $\{x_j\}$ is subnet that converges to $y$, there exists $k\in J$ such that $\scrv_k\succeq \scrw$ and $x_{k}\in W_0$. We have $x_k\in \im( f_{\scrv_k,i_{\scrv_k}})\subseteq U_{\scrv_k,i_{\scrv_k}}\subseteq W$ for some $W\in\scrw$ and thus $\im( f_{\scrv_k,i_{\scrv_k}})\subseteq U_{\scrv_k,i}\subseteq \St(W_0,\scrw)\subseteq Y$. However, for every $\scrv\in O_F(X) $, $f_{\scrv,i_{\scrv}}$ is not null-homotopic in $X$. Thus, since $Y$ represents an arbitrary neighborhood of $y$, $X$ is not semilocally $\pi_n$-trivial at $y$. By assumption, we must have $x=y$. Since $\{x_j\}\to x$, the same argument, but where $Y$ is replaced by $V'$, shows that there exists sufficiently refined $\scrv$ for which $\im(f_{\scrv,i_{\scrv}})\subseteq V'$; a contradiction. Since the claim is proved, there exists $\scru_1\succeq\scru_0$ in $O_F(X)$ such that whenever $\scrv\succeq\scru_1$, we have $\im(f_{\scrv,i})\subseteq V'$ for all $i\in\{1,2,\dots,m_{\scrv}\}$.

Fix $\scrv\succeq\scru_1$ in $O_F(X)$. For all $i\in\{1,2,\dots,m_{\scrv}\}$, we may find a path $\beta_{\scrv,i}:\ui\to V$ from $x$ to $f_{\scrv,i}(d_0)$. Since $X$ is simply connected, we have $[\alpha_{\scrv,i}\ast f_{\scru,i}]=[\beta_{\scrv,i}\ast f_{\scru,i}]$ for all $i$. Thus $g$ is represented by $\prod_{i=1}^{m_{\scrv}}\beta_{\scrv,i}\ast f_{\scrv,i}$, which has image in $V$.
\end{proof}

\begin{remark}[Topologies on homotopy groups]
Given a group $G$ and a collection of subgroups $\{N_j\mid j\in J\}$ of $G$ such that for all $j,j'\in J$, there exists $k\in J$ such that $N_k\subseteq N_j\cap N_{j'}$, we can generate a topology on $G$ by taking the set $\{gN_j\mid j\in J,g\in G\}$ of left cosets as a basis. We can apply this to both the collection of Spanier subgroups $\span(\scru,x_0)$ and the collection of kernels $\ker(p_{\scru \#})$ to define two natural topologies on $\pi_n(X,x_0)$.
\begin{enumerate}
\item The \textit{Spanier topology} on $\pi_n(X,x_0)$ is generated by the left cosets of Spanier groups $\pi_n(\scru,x_0)$ for $\scru\in O(X)$.
\item The \textit{shape topology} on $\pi_n(X,x_0)$ is generated by left cosets of the kernels $\ker(p_{\scru\#})$ where $(\scru,U_0)\in \Lambda$. Equivalently, the shape topology is the initial topology with respect to the map $\Psi_{n}$ where the groups $\pi_n(|N(\scru)|,U_0)$ are given the discrete topology and $\check{\pi}_n(X,x_0)$ is given the inverse limit topology.
\end{enumerate}
Lemma \ref{inclusionlemma1} ensures the Spanier topology is always finer than the shape topology. Lemma \ref{secondinclusionlemma} then implies that, whenever $X$ is paracompact, Hausdorff, and $LC^{n-1}$, the two topologies agree. Moreover, $\pi_n(X,x_0)$ is Hausdorff in the shape topology if and only if $X$ is $\pi_n$-shape injective.
\end{remark}

\section{When is $\Psi_n$ an isomorphism?}\label{section6isomorphism}

It is a result of Kozlowski-Segal \cite{KSvietoris} that if $X$ is paracompact Hausdorff and $LC^n$, then $\Psi_n:\pi_n(X,x)\to \check{\pi}_n(X,x)$ is an isomorphism. This result was first proved for compact metric spaces in \cite{Kuperburg}. The assumption that $X$ is $LC^n$ assumes that small maps $S^n\to X$ may be contracted by small null-homotopies. However, if $\bbe_n$ is the $n$-dimensional earring space, then the cone $C\bbe_{n}$ is $LC^{n-1}$ but not $LC^{n}$. However, $C\bbe_n$ is contractible and so $\Psi_n$ is an isomorphism of trivial groups. Certainly, many other examples in this range exist. Our Spanier group-based approach allows us to generalize Kozlowksi-Segal's theorem in a way that includes this example by removing the need for ``small" homotopies in dimension $n$. For simplicity, we will sometimes suppress the pointedness of open covers and simply write $\scru$ for elements of $\Lambda$.

\begin{lemma}\label{ontolemma}
Let $n\geq 1$. Suppose that $X$ is paracompact, Hausdorff, and $LC^{n-1}$. If $([f_{\scru}])_{\scru\in\Lambda}\in \check{\pi}_1(X,x_0)$, then for every $\scru\in\Lambda$, there exists $[g]\in \pi_n(X,x)$ such that $(p_{\scru})_{\#}([g])=[f_{\scru}]$.
\end{lemma}

\begin{proof}
With $(\scru,U_0)\in \Lambda$ and $p_{\scru}$ fixed, consider a representing map $f_{\scru}:(|\partial \Delta_{n+1}|,{\bf o})\to (|N(\scru)|,U_0)$. Let $\scru '=\{p_{\scru}^{-1}(\st(U,|N(\scru)|))\mid U\in\scru\}$. Since $p_{\scru}^{-1}(\st(U,|N(\scru)|))\subseteq U$ for all $U\in\scru$, we have $\scru\preceq \scru '$. Applying Lemmas \ref{paracompactlemma} and \ref{refinementlemma2} we can choose the following sequence of refinements of $\scru '$. First, we choose a star refinement $\scru '\preceq_{\ast\ast}\scrw$ so that for every $W\in\scrw$, there exists $U'\in\scru '$ such that $\St(W,\scrw)\subseteq U'$. In this case, we can choose the projection map $p_{\scru '\scrw}:|N(\scrw)|\to |N(\scru ')|$ so that if $p_{\scru '\scrw}(W)= U'$ on vertices, then $\St(W,\scrw)\subseteq U'$ in $X$. This choice will be important near the end of the proof.

To construct $g$, we must take further refinements. First, we choose a sequence of a refinements \[\scrw=\scrw_{n}\preceq_{\ast}\scrv_{n}\preceq_{\ast}^{n-1}\scrw_{n-1}\preceq_{\ast}\cdots \preceq_{\ast}^{2}\scrw_{2}\preceq _{\ast} \scrv_{2}\preceq_{\ast}^{1} \scrw_{1}\preceq_{\ast} \scrv_{1}\preceq_{\ast}^{0} \scrw_{0}\]
followed by one last refinement $\scrw_{0}\preceq_{\ast} \scrv_0=\scrv$. Let $V_0\in\scrv$ be any set containing $x_0$ and recall that since $X$ is paracompact Hausdorff $(\scrv,V_0)\in \Lambda$. For some choice of canonical map $p_{\scrv}$, we have $p_{\scrv}^{-1}(\st(V,N(\scrv)))\subseteq V$ for all $V\in\scrv$. 

Recall that we have assumed the existence of a map $f_{\scrv}:(\partial \Delta_{n+1},{\bf o})\to (|N(\scrv)|,V_0)$ such that $p_{\scru\scrv\#}([f_{\scrv}])=[f_{\scru}]$. Set $Y_V=f_{\scrv}^{-1}(\st(V,N(\scrv)))$ so that $\scry=\{Y_V\mid V\in\scrv\}$ is an open cover of $\partial\Delta_{n+1}$. As before, we find a simplicial approximation for $f_{\scrv}$. Find $m\in\bbn$ such that the star $\st(a,\sd^{m}\partial\Delta_{n+1})$ of each vertex $a$ in $\sd^{m}\partial\Delta_{n+1}$ lies in a set $Y_{V_a}\in\scry$ for some $V_a\in \scrv$. Since $f_{\scrv}({\bf o})=V_0$, we may take $V_{{\bf o}}=V_0$. The assignment $a\mapsto V_a$ on vertices extends to a simplicial approximation $f':\sd^{m}\partial\Delta_{n+1}\to |N(\scrv)|$ of $f_{\scrv}$, i.e. a simplicial map $f'$ such that \[f_{\scrv}(\st(a,\sd^{m}\partial\Delta_{n+1}))\subseteq \st(f'(a),|N(\scrv)|)=\st(V_a,|N(\scrv)|)\] for each vertex $a$. 

We begin to define $g$ with the constant map $\{{\bf o}\}\to X$ sending ${\bf o}$ to $x_0$. In preparation for applications of Lemma \ref{recursionlemma}, set $K=\sd^m\partial \Delta_{n+1}$ and $L=\{{\bf o}\}$ so that $K[k]=K_k$. First, we define a map $g_0:M[0]\to X$ by picking, for each vertex $a\in K_0$, a point $g_0(a)\in V_a$. In particular, set $g_0({\bf o})=x_0$. This choice is well defined since we have
$p_{\scrv}(x_0)=  V_0\in \st(V_{{\bf o}},N(\scrv))$ and thus $g_0({\bf o})=x_0\in p_{\scrv}^{-1}(\st(V_{{\bf o}},N(\scrv)))\subseteq V_{{\bf o}}$. Note that $f'$ maps every simplex $\sigma=[a_0,a_1,\dots,a_k]$ of $K$ to the simplex of $|N(\scrv)|$ spanned by $\{V_{a_i}\mid 0\leq i\leq k\}$. By definition of the nerve, we have $\bigcap\{V_{a_i}\mid 0\leq i\leq k\}\neq \emptyset$. Pick a point $x_{\sigma}\in \bigcap\{V_{a_i}\mid 0\leq i\leq k\}$. By our initial choice of refinements, we have $\scru_0\preceq_{\ast} \scrv$. If $\sigma=[a_0,a_1,\dots,a_{n}]$ is a $n$-simplex of $K$, then $\St(x_{\sigma},\scrv)\subseteq U_{0,\sigma}$ for some $U_{0,\sigma}\in\scru_0$. In particular $\{g_0(a_i)\mid 0\leq i\leq n+1\}\subseteq \bigcup\{V_{a_i}\mid 0\leq i\leq n\}\subseteq U_{0,\sigma}$. Thus $g_0$ maps the $0$-skeleton of $\sigma$ into $U_{0,\sigma}$. If ${\bf o}\in \sigma$, then $g_0({\bf o})\in p_{\scrv}^{-1}(\st(V_{{\bf o}},N(\scrv)))\subseteq V_{{\bf o}}\subseteq U_{0,\sigma}$. Hence, for every $n$-simplex $\sigma$ of $K$, we have $g_0(\sigma \cap M[0])\subseteq U_{0,\sigma}$.

We are now in the situation to recursively apply Lemma \ref{recursionlemma}. This is similar to the application in the proof of Lemma \ref{secondinclusionlemma} with the dimension $n+1$ shifted down by one so we omit the details. Recalling that $M[n]=\sd^m\partial \Delta_{n+1}$, we obtain an extension $g:K=M[n]\to X$ of $g_0$ such that for every $n$-simplex $\sigma$ of $K$, we have $g(\sigma)\subseteq W_{\sigma}$ for some $W_{\sigma}\in \scrw=\scru_n$.

With $g$ defined, we seek show that $f_{\scru}\simeq p_{\scru}\circ g$. Since $f'\simeq f_{\scrv}$ (by simplicial approximation), $p_{\scru\scrv}\simeq p_{\scru \scru '}\circ p_{\scru '\scrw}\circ p_{\scrw\scrv}$ (for any choice of projection maps), and $p_{\scru\scrv}\circ f_{\scrv}\simeq f_{\scru}$ (for any choice of projection $p_{\scru\scrv}$), it suffices to show that $p_{\scru\scru '}\circ p_{\scru '\scrw}\circ p_{\scrw\scrv}\circ f'\simeq p_{\scru}\circ g$. We do this by proving that the simplicial map $F=p_{\scru\scru '}\circ p_{\scru '\scrw}\circ p_{\scrw\scrv}\circ f':K\to |N(\scru)|$ is a simplicial approximation for $p_{\scru}\circ g$. Recall that this can be done by verifying the ``star-condition" $p_{\scru}\circ g(\st(a,K))\subseteq \st(F(a),|N(\scru)|)$ for any vertex $a\in K$ \cite[Ch.2 \S 14]{Mu84}. Since $n\geq 1$, we have $\scrw\preceq_{\ast\ast}\scrv$. Hence, just like our choice of $p_{\scru '\scrw}$, we may choose $p_{\scrw\scrv}$ so that whenever $p_{\scrw\scrv}(V)=W$, then $\St(V,\scrv)\subseteq W$. Also, we choose $p_{\scru\scru '}$ to map $p_{\scru}^{-1}(\st(U,|N(\scru)|))\mapsto U$ on vertices.

Fix a vertex $a_0\in K$. To check the star-condition, we'll check that $p_{\scru}\circ g(\sigma)\subseteq \st(F(a_0),|N(\scru)|)$ for each $n$-simplex $\sigma$ having $a_0$ as a vertex. Pick an $n$-simplex $\sigma=[a_0,a_1,\dots,a_n]\subseteq K$ having $a_0$ as a vertex. Recall that $f'(a_i)=V_{a_i}$ for each $i$. Set $p_{\scrw\scrv}(V_{a_i})=W_{i}$ and $p_{\scru '\scrw}(W_{i})=p_{\scru}^{-1}(\st(U_i,|N(\scru)|))\in \scru '$ for some $U_i\in\scru$. Then $F(a_i)=U_i$ for all $i$. It now suffices to check that $p_{\scru}\circ g(\sigma)\subseteq  \st(U_{0},|N(\scru)|)$. Recall that by our choice of $p_{\scru '\scrw}$, we have $\St(W_{0},\scrw)\subseteq p_{\scru}^{-1}(\st(U_0,|N(\scru)|))$. Thus it is enough to check that $g(\sigma)\subseteq \St(W_{0},\scrw)$. By construction of $g$, we have $g(\sigma)\subseteq W_{\sigma}$ for some $W_{\sigma}\in\scrw$. Since $g(a_0)\in W_0\cap W_{\sigma}$, we have $g(\sigma)\subseteq W_{\sigma}\subseteq \St(W_0,\scrw)$, completing the proof.
\end{proof}

Finally, we prove our second result, Theorem \ref{mainresult2}.

\begin{proof}[Proof of Theorem \ref{mainresult2}]
Since $X$ is paracompact, Hausdorff, $LC^{n-1}$, we have $\span(X,x_0)=\ker(\Psi_n)$ by Theorem \ref{maintheorem}. Since $X$ is semilocally $\pi_n$-trivial, we have $\span(\scru,x_0)=1$ for some $\scru\in\Lambda$. It follows that $\Psi_n$ is injective. Moreover, by Lemma \ref{secondinclusionlemma}, we may find $\scrv\in\Lambda$ with $\ker(p_{\scrv\#})\subseteq \span(\scru,x_0)$. Thus $p_{\scrv\#}:\pi_n(X,x_0)\to \pi_n(|N(\scrv)|,V_0)$ is injective. Let $([f_{\scru}])_{\scru\in\Lambda}\in \check{\pi}_n(X,x_0)$. By Lemma \ref{ontolemma}, for each $\scru\in\Lambda$, there exists $[g_{\scru}]\in \pi_n(X,x_0)$ such that $p_{\scru}([g_{\scru}])=[f_{\scru}]$. If $\scrv\preceq\scrw$, then we have \[p_{\scrv\#}([g_{\scrv}])=[f_{\scrv}]=p_{\scrv\scrw\#}([f_{\scrw}])=p_{\scrv\scrw\#}\circ p_{\scrw\#}([g_{\scrw}])=p_{\scrv\#}([g_{\scrw}]).\]
Since $p_{\scrv\#}$ is injective, it follows that $[g_{\scrw}]=[g_{\scrv}]$ whenever $\scrv\preceq\scrw$. Setting $[g]=[g_{\scrv}]$ gives $\Psi_n([g])=([f_{\scru}])_{\scru\in\Lambda}$. Hence, $\Psi_n$ is surjective.
\end{proof}

\section{Examples}

\begin{example}\label{examplesuspensions}
Fix $n\geq 2$. When $X$ is a metrizable $LC^{n-1}$ space, the cone $CX$ and unreduced suspension $SX$ are $LC^{n-1}$ and semilocally $\pi_n$-trivial but need not be $LC^{n}$. This occurs in the case $X=\bbe_n$ or if $X=Y\vee \bbe_n$ where $Y$ is a CW-complex. In such cases, $\Psi_n:\pi_n(SX)\to\check{\pi}_n(SX)$ is an isomorphism. One point unions of such cones and suspensions, e.g. $CX\vee CY$ or $CX\vee SY$ also meet the hypotheses of Theorem \ref{mainresult2} (checking this is fairly technical \cite{BrazSequentialnconn}) but need not be $LC^n$.
\end{example}

\begin{example}
The converse of Theorem \ref{mainresult2} does not hold. For $n\geq 2$, $\bbe_n$ is $LC^{n-1}$ but is not semilocally $\pi_n$-trivial at the wedgepoint $x_0$. However, $\Psi_{n}:\pi_n(\bbe_n,x_0)\to \check{\pi}_n(\bbe_n,x_0)$ is an isomorphism where both groups are canonically isomorphic to $\bbz^{\bbn}$ \cite{EK00higher}. Additionally, for the infinite direct product $\prod_{\bbn}S^n$, $\Psi_{k}:\pi_k(\prod_{\bbn}S^n,x_0)\to \check{\pi}_k(\prod_{\bbn}S^n,x_0)$ is an isomorphism for all $k\geq 1$ even though $\prod_{\bbn}S^n$ is not $LC^{k-1}$ when $k-1\geq n$.
\end{example}

\begin{example}\label{examplemappingtoruscone}
We can also modify the mapping torus $M_f$ from Example \ref{mappingtorusexample} so that $\Psi_n$ becomes an isomorphism (recall that $n\geq 2$ is fixed). Let $X=M_f\cup C\bbe_n$ be the mapping cone of the inclusion $\bbe_n\to M_f$. For the same reason $M_f$ is $LC^{n-1}$, the space $X$ is $LC^{n-1}$. Moreover, if $U$ is a neighborhood of $\alpha(t)$ that deformation retracts onto a homeomorphic copy of $\bbe_n$, then any map $S^n\to U$ may be freely homotoped ``around" the torus and into the cone. It follows that $X$ is semilocally $\pi_n$-trivial. We conclude from Theorem \ref{mainresult2} that $\Psi_n:\pi_n(X)\to \check{\pi}_n(X)$ is an isomorphism. Since sufficiently fine covers of $X$ always give nerves homotopy equivalent to $S^1\vee S^{n+1}$, we have $\check{\pi}_n(X)=0$. Thus $\pi_n(X)=0$.
\end{example}

\begin{example}\label{examplenotuvnminusone}
Let $n\geq 2$ and $X=\bbe_1\vee S^n$ (see Figure \ref{fig3}). Note that because $\bbe_1$ is aspherical \cite{CCZaspherical,CFhigher}, $X$ is semilocally $\pi_n$-trivial. However, $X$ is not $LC^{1}$ because it has $\bbe_1$ as a retract. It is shown in \cite{BrazSequentialnconn} that $\pi_n(X)\cong \bigoplus_{\pi_1(\bbe_1)}\pi_n(S^n)\cong \bigoplus_{\pi_1(\bbe_1)}\bbz$ and that $\Psi_{n}:\pi_n(X)\to \check{\pi}_n(X)$ is injective. In particular, we may represent elements of $\pi_n(X)$ as finite-support sums $\sum_{\beta\in\pi_1(\bbe_1)}m_{\beta}$ where $m_{\beta}\in\bbz$. We show that $\Psi_{n}$ is not surjective.

Identify $\pi_1(X)$ with $\pi_1(\bbe_1)$ and recall from \cite{Edafreesigmaproducts} that we can represent the elements of $\pi_1(\bbe_1)$ as countably infinite reduced words indexed by a countable linear order (with a countable alphabet $\beta_1,\beta_2,\beta_3,\dots$). Here $\beta_j$ is represented by a loop $S^1\to \bbe_1$ going once around the $j$-th circle. Let $X_j$ be the union of $S^n$ and the largest $j$ circles of $\bbe_1$ so that $X=\varprojlim_{j}X_j$. Identify $\pi_1(X_j)$ with the free group $F_j$ on generators $\beta_1,\beta_2,\dots \beta_j$ and note that $\pi_n(X_j)\cong \bigoplus_{F_j}\bbz$. Thus we may view an element of $\pi_n(X_j)$ as a finite-support sums $\sum_{w\in F_j}m_{w}$ of integers indexed over reduced words in $F_j$. Let $d_{j+1,j}:F_{j+1}\to F_j$ be the homomorphism that deletes the letter $\beta_{j+1}$. Consider the inverse limit $\check{\pi}_1(X)=\varprojlim_{j}(F_j,d_{j+1,j})$. The map $X\to X_j$ that collapses all but the first $j$-circles of $\bbe_1$ induces a homomorphism $d_j:\pi_1(X)\to F_j$. There is a canonical homomorphism $\phi:\pi_1(X)\to \check{\pi}_1(X)=\varprojlim_{j}(F_j,d_{j+1,j})$ given by $\phi(\beta)=(d_1(\beta),d_2(\beta),\dots)$, which is known to be injective \cite{MM} but not surjective. For example, if $x_k=\prod_{j=1}^{k}[\beta_1,\beta_j]$, then $(x_1,x_2,x_3,x_4,\dots)$ is an element of $\check{\pi}_1(X)$ not in the image of $\phi$.

The bonding map $b_{j+1,j}:\pi_n(X_{j+1})\to \pi_n(X_j)$ sends a sum $\sum_{w\in F_{j+1}}m_{w}$ to $\sum_{v\in F_{j}}p_v$ where $p_v=\sum_{d_{j+1,j}(w)=v}m_{w}$. Similarly, projection map $b_j:\pi_n(X)\to \pi_n(X_j)$ sends the sum $\sum_{\beta\in\pi_1(X)}n_{\beta}$ to $\sum_{v\in F_{j}}m_v$ where $m_v=\sum_{d_j(\beta)=v}m_{\beta}$. Let $y_j\in \pi_n(X)$ be the sum whose only non-zero coefficient is the $x_j$-coefficient, which is $1$. Since $d_{j+1,j}(x_{j+1})=x_j$, it's clear that $(y_1,y_2,y_3,\dots)\in \check{\pi}_n(X)$. Suppose $\Psi_n(\sum_{\beta}m_{\beta})=(y_1,y_2,y_3,\dots)$. Writing $\sum_{\beta}m_{\beta}$ as a finite sum $\sum_{i=1}^{r}m_{\beta_i}$ for non-zero $m_{\beta_i}$, we must have $\sum_{d_j(\beta_i)=x_j}m_{\beta_i}=1$ for all $j\in\bbn$. Since there are only finitely many $\beta_i$ involved, there must exist at least one $i$ for which $d_j(\beta_i)=x_j$ for infinitely many $j$. For such $i$, we have $\phi(\beta_i)=(x_1,x_2,x_3,\dots)$, which, as mentioned above, is impossible. Hence $\Psi_{n}$ is not surjective.
\end{example}

\begin{figure}[H]
\centering \includegraphics[height=2in]{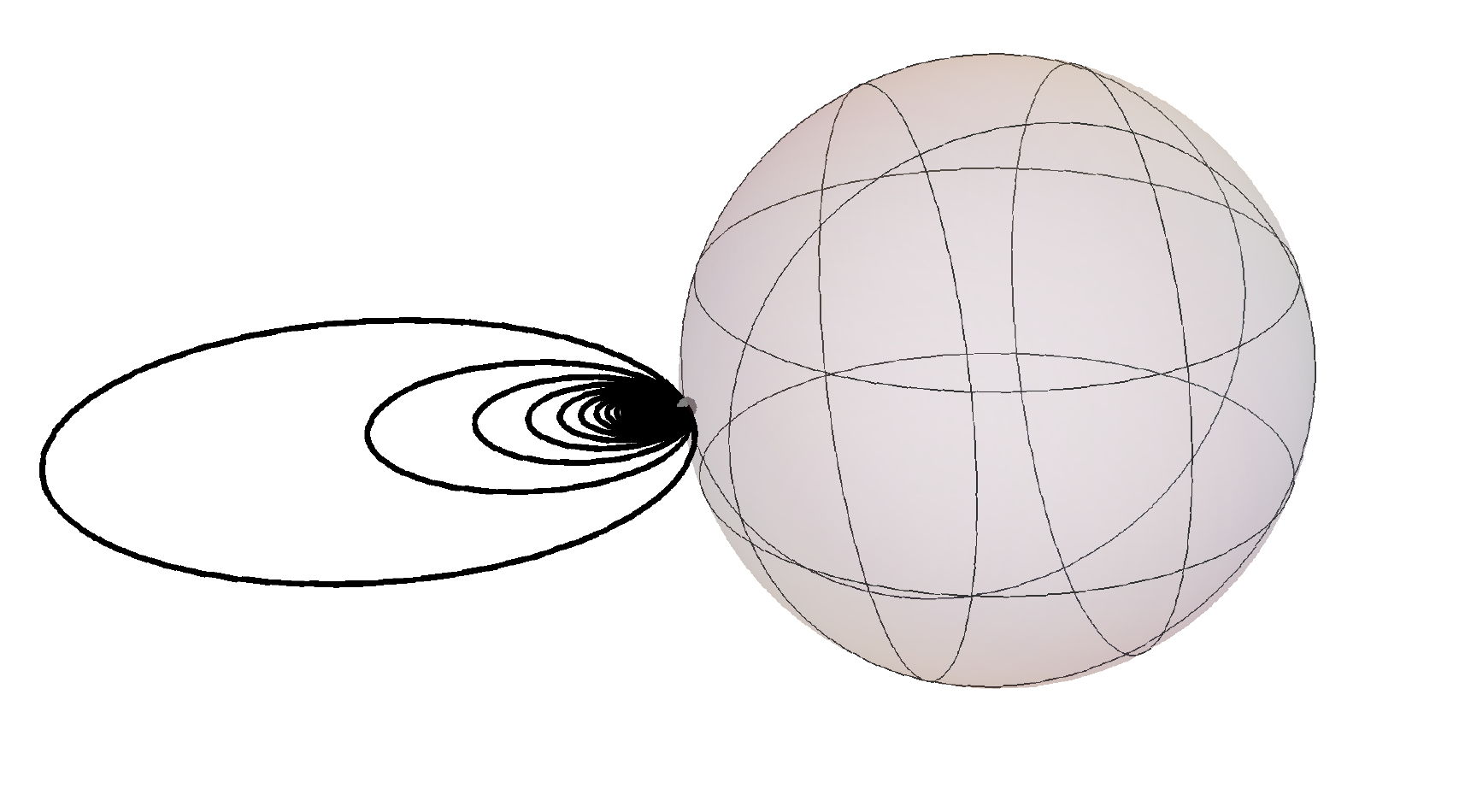}
\caption{\label{fig3}The one point union $\bbe_1\vee S^2$. }
\end{figure}

The previous example shows why we cannot do away with the $LC^{n-1}$ hypothesis in Theorem \ref{mainresult2}. Since we weakened the hypothesis from \cite{KSvietoris} in dimension $n$ and no hypothesis in dimension $n$ is required for Theorem \ref{maintheorem}, one might suspect that we might be able to do away with the dimension $n$ hypothesis completely. The next example, which is a higher analogue of the harmonic archipelago \cite{BS,CHM,KarimovRepovs} shows why this is not possible.

\begin{example}\label{examplenotsemilocallypintrivial}
Let $n\geq 2$ and $\ell_j:S^n\to \bbe_n$ be the inclusion of the $j$-th $n$-sphere in $\bbe_n$. Let $X$ be the space obtained by attaching $(n+1)$-cells to $\bbe_n$ using the attaching maps $\ell_j$. Since $\bbe_n$ is $LC^{n-1}$, it follows that $X$ is $LC^{n-1}$. However, $X$ is not semilocally $\pi_n$-trivial at the wedgepoint $\bf{o}$ of $\bbe_n$. Indeed, the infinite concatenation maps $\prod_{j\geq k}\ell_j=\ell_{k}\cdot\ell_{k+1}\cdots$ are not null-homotopic (using a standard argument that works for the harmonic archipelago) but are all homotopic to each other. Thus, $\pi_n(X,{\bf o})\neq 0$. However, for sufficiently fine open covers $\scru\in O(X)$, $|N(\scru)|$ is homotopy equivalent to a wedge of $(n+1)$-spheres and thus $\check{\pi}_n(X,{\bf o})=0$. Therefore, despite $X$ being $LC^{n-1}$, $\Psi_{n}$ is not an isomorphism. In fact, $\pi_n(X,{\bf o})=\span(X,{\bf o})=\ker(\Psi_n)$. The reader might also note that since $\bbe_{n-1}$ is $(n-1)$-connected and $\pi_n(\bbe_n,{\bf o})\cong H_n(\bbe_n)\cong \bbz^{\bbn}$, $X$ will also be $(n-1)$-connected. A Meyer-Vietoris Sequence argument similar to that in \cite{KarimovRepovs} can then be used to show $\pi_n(X,{\bf o})\cong H_n(X)\cong \bbz^{\bbn}/\oplus_{\bbn}\bbz$.
\end{example}

\end{document}